\documentclass[12pt,a4paper]{amsart}
\pdfoutput=1

\usepackage[T1]{fontenc}
\usepackage{lmodern}
\usepackage[tracking]{microtype}
\usepackage[utf8]{inputenc}

\usepackage[backend=biber,
            style=numams, 
            sorting=nyt, 
            isbn=false,  
            backref=true, 
            maxbibnames=4]{biblatex}
\usepackage{amsmath,amssymb, amsthm}
\usepackage{mathtools}

\usepackage[colorlinks, 
            linkcolor=blue, 
            citecolor=blue, 
            pdfa]{hyperref}

\usepackage[hcentering, hscale=0.64, vscale=0.70 ]{geometry}

\newtheorem{introthm}{Theorem}

\swapnumbers
\newtheorem{lemma}{Lemma}[section]
\newtheorem{cor}[lemma]{Corollary}

\theoremstyle{remark}
\newtheorem{remark}[lemma]{Remark}

\renewcommand{\phi}{\varphi}
\renewcommand{\theta}{\vartheta}
\newcommand{\eps}{\varepsilon}

\renewcommand{\geq}{\geqslant}
\renewcommand{\leq}{\leqslant}

\newcommand{\reals}{\mathbb{R}}

\newcommand{\iso}{\cong}    
\newcommand{\Sym}[1]{\operatorname{Sym}(#1)}

\DeclarePairedDelimiter{\card}{\lvert}{\rvert}  
\DeclarePairedDelimiter{\erz}{\langle}{\rangle} 

\DeclareMathOperator{\GL}{GL}
\DeclareMathOperator{\AGL}{AGL}
\DeclareMathOperator{\Aut}{Aut}
\DeclareMathOperator{\Z}{\mathbf{Z}}     
\DeclareMathOperator{\C}{\mathbf{C}}     
\DeclareMathOperator{\N}{\mathbf{N}}
\DeclareMathOperator{\mat}{\mathbf{M}}      
\DeclareMathOperator{\conv}{conv} 
\DeclareMathOperator{\Id}{id}       

\hypersetup{pdfauthor = {Barbara Baumeister, Frieder Ladisch},
            pdftitle={A property of the Birkhoff polytope},
            pdfkeywords = {Birkhoff polytope, representation polytope,
                      permutation polytope, combinatorial symmetry}
           }

\begin{document}
\title{A property of the Birkhoff polytope}
\author{Barbara Baumeister}
\address{Universität Bielefeld\\
         Postfach 100131\\ 
         33501 Bielefeld\\
         Germany}
\email{b.baumeister@math.uni-bielefeld.de}
\author{Frieder Ladisch}
\address{Universität Rostock\\
         Institut für Mathematik\\
         18051 Rostock\\
         Germany}
\email{frieder.ladisch@uni-rostock.de}

\begin{abstract}
  The Birkhoff polytope $B_n$ is the convex hull
  of all $n\times n$ permutation matrices in
  $\mathbb{R}^{n\times n}$.  
  We compute the combinatorial symmetry group 
  of the Birkhoff polytope.
  
  A representation polytope is 
  the convex hull of some finite matrix group
  $G\leq \operatorname{GL}(d,\mathbb{R})$.
  We show that the group of permutation matrices
  is essentially the only finite matrix group
  which yields a representation polytope 
  with the same face lattice as the Birkhoff polytope.  
\end{abstract}

\maketitle

\section{Introduction}
Let $P\colon G= S_n \to \GL(n,\reals)$ be the 
standard permutation representation of 
the symmetric group~$S_n$ on $n$ letters.
The \emph{Birkhoff polytope} $B_n$ is by definition
the convex hull of all permutation matrices of size $n\times n$:
\[ B_n := \conv\{ P(\sigma) \mid \sigma\in S_n\}.
\]
In this note, we prove a conjecture of
Baumeister, Haase, Nill and Paffenholz~\cite[Conjecture~5.3]{BHNP09}
on the uniqueness of the Birkhoff polytope among
permutation polytopes.
In fact, we prove a slightly stronger result.

To state the result, we need the following notation.
Let $D\colon G\to \GL(d,\reals)$
be a representation over the reals.
The corresponding \emph{representation polytope},
$P(D)$, is the convex hull of the image of $D$:
\[ P(D):= \conv\{ D(g) \mid g\in G
               \}.
\]
If $D$ is a permutation representation, then the 
representation polytope is called a 
\emph{permutation polytope}.

Two representations 
$D_i\colon G_i \to \GL(d_i,\reals)$
(where $i=1$, $2$) are called 
\emph{effectively equivalent} if there is 
a group isomorphism $\phi\colon G_1\to G_2$
such that $D_1$ and $D_2\circ \phi$ are 
\emph{stably equivalent},
which means that $D_1$ and $D_2\circ \phi$ 
have the same nontrivial irreducible constituents 
(not necessarily occurring with the same multiplicities).
The representation polytopes of effectively representations are
affinely isomorphic
\cite[\S~2]{BHNP09} \cite[Theorem~2.4]{BaumeisterGrueninger15}.
The converse is not true, for example, 
when $D$ is the regular representation of a group,
then $P(D)$ is a simplex of dimension $\card{G}-1$.
Thus groups that are not even isomorphic as abstract groups,
may yield affinely equivalent representation polytopes.

From this viewpoint, the next result is somewhat surprising.
Recall that two polytopes 
$P$ and $Q$ are
\emph{combinatorially equivalent}
if there is a bijection between the vertices of $P$
and the vertices of $Q$ which maps faces of $P$
onto faces of $Q$.
Affinely equivalent polytopes are combinatorially equivalent,
but not conversely.

\begin{introthm}\label{main:birkhoff_unique}
  Let $D\colon G\to \GL(d,\reals)$ be 
  a faithful representation such that
  the representation polytope~$P(D)$ is combinatorially equivalent
  to the Birkhoff polytope~$B_n$.
  Then either $n=3$ and $G$ is cyclic of order $6$,
  or $D$ and the standard permutation representation
  $P\colon S_n \to \GL(n,\reals)$ 
  are effectively equivalent  
  (in particular, $G\iso S_n$).
\end{introthm}

In the exceptional case $n=3$ and $G$ cyclic, it is easy to see that
$D$ is not stably equivalent to a permutation representation.
It follows also from the classification of permutation polytopes
in small dimensions~\cite[Theorem~4.1]{BHNP09}
that $B_3$ is not combinatorially equivalent 
to any other permutation polytope.
In particular, Theorem~\ref{main:birkhoff_unique} 
answers ~\cite[Conjecture~5.3]{BHNP09} in the positive.

To prove Theorem~\ref{main:birkhoff_unique}, 
we use the determination of the 
combinatorial symmetry group of the Birkhoff polytope,
which may be of interest in its own right:

\begin{introthm}\label{main:birkhoff_symgrp}
  For every combinatorial symmetry $\alpha$ of 
  the Birkhoff polytope 
  there are $\sigma$, $\tau\in S_n$ and $\eps\in \{\pm 1\}$
  such that $\alpha(\pi)= \sigma \pi^{\eps} \tau$
  for all $\pi\in S_n$.  
  Every combinatorial symmetry comes from an isometry of
  the space of $n\times n$ matrices over $\reals$.
\end{introthm}

As we will explain below, this means that for $n\geq 3$,
the combinatorial symmetry group of the Birkhoff polytope
is isomorphic to the wreath product 
$S_n \wr C_2 = (S_n \times S_n)\rtimes C_2$.

Although not difficult, this result seems not to be in the literature
yet.
There are, however, two different published proofs 
that the above maps 
are all the linear maps preserving the Birkhoff 
polytope~\cite{LiSpiZobin04,LiTamTsing02}.
Since every linear or affine symmetry of a polytope 
induces a combinatorial symmetry,
Theorem~\ref{main:birkhoff_symgrp} 
is actually stronger than the old result.
As one would expect, our proof of Theorem~\ref{main:birkhoff_symgrp}
depends on the well known description of 
the facets and thus the combinatorial structure
of the Birkhoff polytope.
On the other hand, the combinatorial structure of 
representation and permutation polytopes in general 
can be quite complicated,
even for cyclic groups, as examples show~\cite{BHNP11pre}.

\section{Preliminaries on permutation actions on a group}
Let $G$ be a finite group.
For each $g\in G$, let $\lambda_g\in \Sym{G}$ 
be left multiplication with $g$
(so $\lambda_g(x)= gx$),
and $\rho_g$ be right multiplication
with $g^{-1}$, that is, $\rho_g(x)= xg^{-1}$.
Thus $g\mapsto \lambda_g$ and $g\mapsto \rho_g$ are 
the left and right regular permutation action.
Also, let $\iota\in \Sym{G}$ be the map that 
inverts elements (so $\iota(x)= x^{-1}$ for all $x\in G$).
Let $\Gamma(G)\leq \Sym{G}$ be the group generated by all these
elements:
\[ \Gamma(G):= \erz{\lambda_g, \rho_g, \iota
                   \mid g\in G}.
\]
To describe $\Gamma(G)$, we need the
\emph{wreath product} $G\wr C_2$ 
of $G$ with a cyclic group~$C_2=\erz{s}$
of order $2$.
Recall that this is the semidirect product of
$G\times G$ with $C_2$, where $s$ acts on 
$G\times G$ by exchanging coordinates:
$(g,h)^s = (h,g)$ for $g$, $h\in G$.
Then:

\begin{lemma}
  If $G$ is not an elementary abelian $2$-group,
  then $\Gamma(G) \iso (G\wr C_2)/Z$,
  where $Z = \{(z,z)\in G \times G \mid z \in \Z(G)\}$.
\end{lemma}

\begin{proof}
  We have that $\lambda(G)$ and $\rho(G)$ centralize each other,
  and $(\lambda_g)^{\iota} = \rho_g$.
  Thus sending $(g,h)\in G\times G$ to
  $\lambda_g \rho_h$ and $s\in C_2 = \{1,s\}$
  to $\iota$ defines a surjective group homomorphism
  $G\wr C_2 \to \Gamma(G)$
  with $Z$ in the kernel.
  
  Suppose $\lambda_g \rho_h = \Id_G$.
  Then $gxh^{-1} = x$ for all $x\in G$.
  Taking $x=1$ yields $g=h$, 
  and it follows that  $g\in \Z(G)$.
  
  Now assume $\lambda_g \rho_h \iota = \Id$.
  Then $gx^{-1}h^{-1} = x$ for all $x\in G$,
  and $x=1$ yields $g=h$.
  Moreover, we have
  $xy = g (xy)^{-1} g^{-1}
      = g y^{-1} g^{-1} \, g x^{-1} g^{-1}
      = y x$
  for all $x$, $y\in G$.
  Thus $G$ must be abelian in this case, and
  $x^{-1} = x$ for all $x\in G$.
    
  So when $G$ is not an elementary abelian $2$-group,
  such an element can not be in the kernel of the action
  of $G\wr C_2$ on $G$.
  This shows the result.
\end{proof}

In the proof of Theorem~\ref{main:birkhoff_unique}, 
we need the fact
that $\Gamma(G)$ contains no pair of commuting,
regular subgroups other than $\lambda(G)$
and $\rho(G)$, when $G = S_n$ and $n\geq 4$.
The exception in Theorem~\ref{main:birkhoff_unique}
for $n=3$
comes from the fact that 
in $\Gamma(S_3)$, we have other pairs of commuting, regular
subgroups, namely $U=V = C_2 \times C_3$
and $U=V= C_3\times C_2$.
Notice that we do not assume that the commuting, regular subgroups
$U$, $V$
of $\Gamma(G)$ have trivial intersection.
If one assumes $U\cap V=1$, one can give a somewhat 
shorter proof that 
$\{U,V\} = \{ \lambda(G), \rho(G)\}$
for almost simple groups $G$,
but we need the stronger statement for the proof of 
Theorem~\ref{main:birkhoff_unique}.

The most elegant and elementary way to prove 
that $\lambda(G)$ and $\rho(G)$ form the only pair
of commuting regular subgroups of $\Gamma(G)$
(when $G= S_n$, $n\geq 4$),
seems to be to use a general argument due to
Chermak and Delgado~\cite{ChermakDelgado89}.
Let $G$ be an arbitrary finite group.
Following Isaacs~\cite[\S~1G]{IsaacsFGT},
we call $m_G(H) := \card{H}\card{\C_G(H)}$
the \emph{Chermak-Delgado measure} of the subgroup
$H\leq G$.

\begin{lemma}\cite[Theorem~1.44]{IsaacsFGT}
  \label{l:cdlatt}
  Let $G$ be a finite  group and let
  $\mathcal{L} = \mathcal{L}(G)$ be the set of subgroups
  for which the Chermak-Delgado measure is as large as 
  possible.
  Then for $H$, $K\in \mathcal{L}$, we have
  $H\cap K\in \mathcal{L}$,
  $\erz{H,K} =HK=KH \in \mathcal{L}$, 
  and $\C_G(H)\in \mathcal{L}$.
\end{lemma}

The \emph{Chermak-Delgado lattice} of $G$ is by definition
the set of all subgroups of $G$ for which 
the Chermak-Delgado measure is maximized.
The last result tells us that this is indeed a sublattice
of the lattice of all subgroups of $G$.
We need the following, which is probably well known:

\begin{cor}
  Any member of the Chermak-Delgado lattice
  of a finite group $G$ is subnormal in $G$.
\end{cor}

\begin{proof}
If $H$ is a member of the Chermak-Delgado lattice of $G$,
then any conjugate $H^g$ is also in the 
Chermak-Delgado lattice,
and so $HH^g = H^g H$ by Lemma~\ref{l:cdlatt}.
But subgroups $H\leq G$ with $HH^g = H^gH$ for all $g\in G$
are subnormal~\cite[Theorem~2.8]{IsaacsFGT}.
\end{proof}

\begin{lemma}\label{l:sn_cent_est}
    Suppose that $G$ is almost simple
    (that is, $G$ has a nonabelian simple
    socle). 
    Then $\card{U}\card{\C_{G}(U)} \leq \card{G}$
    for any subgroup $U\leq G$,
    and equality holds if and only if
    $U = \{1\}$ or $U=G$.
    In particular, this holds for 
    $G= S_n$, $n\geq 5$.
    The conclusion is also true for $G=S_4$.
\end{lemma}

\begin{proof}
  Suppose that $1\neq H$ is a member of
  the Chermak-Delgado lattice.
  Then $H$ is subnormal and thus contains the nonabelian
  simple socle of $G$.
  It follows that $\Z(H)=1 = H \cap \C_G(H)$.
  Since $\C_G(H)$ is also a member of the Chermak-Delgado lattice,
  we must have $\C_G(H)=1$.
  Since $\card{H}\card{\C_G(H)} = \card{H} \leq \card{G}$
  was supposed to be maximal possible, we see that
  $H=G$.
  Thus the Chermak-Delgado lattice contains exactly the groups
  $1$ and $G$ itself,
  and the first assertion follows.
  The case $G=S_4$ is a simple verification.
\end{proof}

We will need the following application (for $G=S_n$):

\begin{lemma}\label{l:snwrc2_reg}
  Let $G$ be a group such that the Chermak-Delgado lattice of $G$
  contains exactly the groups $1$ and $G$.
  Then $\lambda(G)$, $\rho(G)$ is the only pair of
  commuting, regular subgroups of\/ $\Gamma(G)$.
\end{lemma}

\begin{proof}
    Notice that $\Z(G)=\{1\}$, since otherwise
    $m_G(\Z(G))=\card{\Z(G)}\card{G} > \card{G} = m_G(1)$.
    Thus $\Gamma(G)\iso G\wr C_2$
    and $\lambda(G)\rho(G) \iso G\times G$.
    
    We first show that a regular subgroup $U$ of 
    $\Gamma(G)$ is contained in the normal subgroup
    $\lambda(G) \rho(G)$.
    Otherwise, $U$
    contains an element $u = \lambda_g \rho_h \iota$
    sending $x\in G$ to $ g x^{-1} h^{-1}$.
    Then $u^2$ sends $x$ to $gh x g^{-1}h^{-1}$,
    and in particular fixes $g$.
    By regularity, we must have $u^2 = \Id_G$. 
    This implies 
    $gh = hg $ and 
    $gh\in \Z(G) = \{1\}$.
    Thus $u$ sends $x$ to $g x^{-1}g$,
    and so fixes $g$, too, 
    which contradicts the regularity. 
    This shows that $U \leq \lambda(G) \rho(G)$.
    
    Since $\lambda(G) \rho(G)\iso G\times G$, we may work
    in $G\times G$ from now on.
    Suppose that $U$ and $V\leq G\times G$ 
    both have size $\card{G}$,
    and commute with each other.
    Let $U_L$ be the projection of $U$ 
    onto the first component, 
    that is, the subgroup of elements
    $g\in G$ such that there is an $h\in G$
    with $(g,h)\in U$.
    Let $U_R$ be the projection of $U$ on the second component.
    With this notation,
    $\C_{G\times G}(U) = \C_{G}(U_L)\times \C_{G}(U_R)$.
    Thus 
    \[ \card{G}^2 
       = \card{U}\card{V}
       \leq \card{U_L}\card{U_R}
            \card{\C_{G}(U_L)}\card{\C_{G}(U_R)}
       \leq \card{G}^2,
    \]
    where the last inequality follows from our assumption
    on the Chermak-Delgado lattice of $G$.
    Thus equality holds, and it follows also that
    $U_L$ and $U_R$ are trivial or the group $G$ itself.
    Since both $U$ and $V$ have size $\card{G}$,
    it follows that
    $\{U,V\} = \{G\times 1, 1\times G\}$.
\end{proof}

\begin{cor}\label{c:normlgamma}
  Let $G$ be a group such that the Chermak-Delgado lattice of $G$
  contains exactly the groups $1$ and $G$.
  Then $\N_{\Sym{G}}(\Gamma(G)) = (\Aut G)\Gamma(G)$.    
\end{cor}

\begin{proof}
  Let $\pi \in \N_{\Sym{G}}(\Gamma(G))$.
  Then $\lambda(G)^{\pi}$ and $\rho(G)^{\pi}$
  are commuting regular subgroups of $\Gamma(G)$,
  and thus 
  $\{\lambda(G)^{\pi},\rho(G)^{\pi}\}
   = \{ \lambda(G), \rho(G)\} $.
  Since $\lambda(G)$ and $\rho(G)$ are conjugate in 
  $\Gamma(G)$, 
  we may assume that $\lambda(G)^{\pi}= \lambda(G)$.
  Thus $\pi \lambda_g \pi^{-1} = \lambda_{\alpha g}$
  for some bijection $\alpha\colon G\to G$.
  Clearly, $\alpha$ is a group automorphism.
  
  As $\lambda(G)$ acts transitively on $G$, we may assume
  $\pi(1) = 1$. 
  But then 
  $\pi(g) = \pi\lambda_g \pi^{-1}(1)
          = \lambda_{\alpha g}(1)
          = \alpha(g)$,
  so $\pi\in \Aut G$.
\end{proof}

The conclusion of this corollary is also true for some other
groups (for example, $G=S_3$),
but not for all groups (for example, $G= S_3 \times S_3$).

\section{The combinatorial symmetry group of the Birkhoff polytope}
Let $D\colon G \to \GL(d,\reals)$ be a faithful representation
and let $P(D) = \conv\{ D(g) \mid g\in G \}$
be the corresponding representation polytope.
Then the vertices of $P(D)$ correspond to the elements of $G$.
We may thus view the affine and combinatorial symmetries
as permutations of $G$ itself.

\begin{lemma}\label{l:reppoltrivsyms}
  Let 
  $D\colon G\to \GL(d,\reals)$
  be a faithful representation
  and $P(D)$ the representation polytope.  
  Then the affine symmetry group $\AGL(P(D))$ 
  as permutation group on $G$ contains 
  $\Gamma(G)$ as defined in the last section.
\end{lemma}

\begin{proof}
  The left multiplications $\lambda_g$ are realized
  by left multiplication with 
  $D(g)$, and the right multiplications $\rho_g$
  by right multiplication with $D(g)^{-1}$.
  If $D$ is an orthogonal representation, then the 
  permutation $g\mapsto g^{-1}$ is realized by transposing 
  matrices, sending $D(g) $ to $D(g)^t = D(g^{-1})$.
  The general case (which we will not need)
  can be reduced to the orthogonal 
  case~\cite[Prop.~6.4]{FrieseLadisch16a}.
\end{proof}

Now let $P\colon G= S_n \to \GL(n,\reals)$ be the 
standard permutation representation of 
the symmetric group~$S_n$, and let
\[ B_n := \conv\{ P(\sigma) \mid \sigma\in S_n\}
\]
be the Birkhoff polytope.
Theorem~\ref{main:birkhoff_symgrp} claims that 
$\Gamma(S_n)$ is the combinatorial symmetry group of
$B_n$.
(The second claim of Theorem~\ref{main:birkhoff_symgrp} is that
these symmetries come from isometries of the matrix space.
This is then clear, since
the symmetries in $\Gamma(S_n)$ even act by permuting coordinates of
the matrices.)

\begin{proof}[Proof of Theorem~\ref{main:birkhoff_symgrp}]
  Recall that the Birkhoff polytope consists of the doubly
  stochastic matrices~\cite[Corollary~1.4.14]{LovaszPlummer86}.
  In particular, for each index pair $(i,j)$,
  the equality
  $a_{ij}=0$ describes a facet of the Birkhoff polytope.
  Thus its facets, as subsets of $S_n$, are given by
  the $n^2$ subsets
  \[ F_{ij} = \{\pi\in S_n \mid \pi(i)\neq j \},
     \quad \text{$i$, $j=1$, $\dotsc $, $n$}.
  \]
  It will be more convenient to work with the complements
  \[ A_{ij}= S_n \setminus F_{ij}
           = \{ \pi \in S_n \mid \pi(i)=j \}
  \] 
  of the facets.
  For $\sigma$, $\tau\in S_n$, 
  we have 
  $\sigma A_{ij} \tau^{-1} = A_{\tau i, \sigma j} $.
  We also have 
  $ A_{ij}^{-1} := \{ \pi^{-1} 
                      \mid \pi \in A_{ij}
                   \}
                 = A_{ji}$.
  Moreover, for $i$, $j$, $k$ and $l$ in $\{1,\dotsc,n\}$
  we have
  \[ \card{A_{ij}\cap A_{kl}}
      =
      \begin{cases}
        (n-1)!, & \text{if } i=k, j=l, \\
        0       & \text{if } i=k, j\neq l,\\
        0       & \text{if } i\neq k, j=l, \\
        (n-2)! \quad & \text{otherwise}.
      \end{cases}
  \]  
  Any combinatorial symmetry $\alpha$ permutes the facets and thus
  the sets $A_{ij}$, and preserves cardinalities of their
  intersections.
  
  Let $\alpha\colon S_n \to S_n$ be an arbitrary
  combinatorial symmetry of the Birkhoff polytope.
  We have to show that $\alpha\in \Gamma(S_n)$,
  the group containing the maps
  $\pi \mapsto \sigma \pi^{\pm 1} \tau^{-1}$.
  After replacing $\alpha$ by $\gamma \circ \alpha$
  for some $\gamma\in \Gamma(S_n)$ of the form
  $\gamma(\pi) = \sigma \pi \tau^{-1}$,
  we may  assume that $\alpha(A_{11})= A_{11}$.
  Then $\card{\alpha(A_{12})\cap A_{11}} = \card{A_{12}\cap A_{11}}=0$,
  and thus either $\alpha(A_{12})= A_{1j}$ for some $j\neq 1$
  or $\alpha(A_{12})=A_{j1}$ for some $j\neq 1$.
  If the latter is the case, we compose $\alpha$
  with the map $\pi\mapsto \pi^{-1}$,
  so we may assume that $\alpha(A_{12})=A_{1j}$.
  
  Multiplying $A_{1j}$ from the left with the transposition
  $(2,j)$ yields the set $A_{12}$,
  and so we can assume that $\alpha(A_{12})= A_{12}$.
  
  Now for $j\geq 3$, the set $\alpha(A_{1j})$ has empty intersection
  with $A_{11}$ and $A_{12}$ and thus
  $\alpha(A_{1j}) \in \{ A_{1k}\mid k\geq 3\}$.
  Thus $\alpha$ induces a permutation 
  $\sigma$ of $\{3,\dotsc, n\}$ defined by
  $\alpha(A_{1j})= A_{1,\sigma j}$.
  Thus $\sigma^{-1} \alpha(A_{1j})= A_{1j}$,
  and we may assume that $\alpha(A_{1j})=A_{1j}$
  for all $j$.
  Similarly, we can assume that
  $\alpha(A_{j1})= A_{j1}$ for all $j$.
  
  Thus, after composing $\alpha$ with suitable elements
  from $\Gamma(S_n)$, 
  we may assume that $\alpha$ leaves each of the sets
  $A_{1j}$ and $A_{j1}$ invariant.
  For $k\geq 2$, $l\geq 2$ we have that
  $A_{kl}$ is the unique set $S$ among the sets
  $A_{ij}$ (with $i\geq 2$, $j\geq 2$)
  such that $S\cap A_{k1}=\emptyset = S\cap A_{1l}$.
  It follows that
  $\alpha(A_{kl})=A_{kl}$ for all $k$, $l$.
  Thus $\alpha$ is the identity.
  It follows that the original $\alpha$ was already in 
  $\Gamma(S_n)$. 
\end{proof}

\section{Characterization of the Birkhoff polytope}
In this section, we prove Theorem~\ref{main:birkhoff_unique}.
We first show the following weaker result.

\begin{lemma}\label{l:snbirk_unique}
  Let $D\colon S_n \to \GL(d,\reals)$ be a representation 
  such that the representation polytope~$P(D)$ is combinatorially
  equivalent to the Birkhoff polytope.
  Then $D$ is effectively equivalent to 
  the standard permutation representation $P$ of $S_n$.  
\end{lemma}

\begin{proof}  
  We have to show that $D$ has the same nontrivial constituents
  as $P$, up to automorphisms of $S_n$.
  Since we can replace $D$ by a stably equivalent representation,
  we may (and do) assume that
  the trivial character is not a constituent of the character of $D$.
  
  A combinatorial isomorphism from the Birkhoff polytope~$B_n$
  onto $P(D)$ sends a vertex $P(g)$ of $B_n$ (where $g\in S_n$)
  to a vertex $D(\alpha(g))$ of $P(D)$,
  where $\alpha\colon S_n \to S_n$ is a permutation of $S_n$.
  Then the map sending
  $\gamma\in \Sym{S_n}$ to 
  $\alpha \circ \gamma\circ \alpha^{-1}$
  is an isomorphism from the combinatorial symmetry group of 
  $B_n$ onto the combinatorial symmetry group of $P(D)$.
  The combinatorial symmetry group of the Birkhoff
  polytope is $\Gamma(S_n)$,
  and the combinatorial symmetry group of $P(D)$ contains 
  $\Gamma(S_n)$ (in its natural action on $P(D)$), 
  by Lemma~\ref{l:reppoltrivsyms}.
  Therefore, the combinatorial symmetry group of $P(D)$
  is just $\Gamma(S_n)$.
  It follows that $\alpha \in \N_{\Sym{S_n}}(\Gamma(S_n))$.
  By Lemma~\ref{l:sn_cent_est}, Corollary~\ref{c:normlgamma}
  applies to $S_n$ and thus 
  $\alpha\in (\Aut S_n)\Gamma(S_n)$.
  After multiplying $\alpha$ with an element of
  $\Gamma(S_n)$, we may thus assume
  $\alpha \in \Aut S_n$.
  Since then $D$ and $D\circ\alpha$ are effectively equivalent,
  we may assume that $\alpha = \Id_{S_n}$.
  This means that the combinatorial isomorphism from $B_n$ onto $P(D)$
  simply sends the vertex $P(g)$ to $D(g)$, for any $g\in S_n$.
  In particular, a subset of $S_n$ corresponds to a face(t)
  of $B_n$ (under $P$) 
  if and only if it corresponds to a face(t)
  of the representation polytope $P(D)$ (under $D$).

    Let $H \leq S_{n}$ be the stabilizer of a point,
    say $n$. (So $H\iso S_{n-1}$.)
    By the description of the facets of $B_n$, we know that
    $S_n \setminus H = \{g \in S_n\mid g(n)\neq n\}$
    corresponds to a facet of $B_n$.
    Thus $D(S_n\setminus H)$ is a facet of $P(D)$.
      
    Let $\phi$ be the character of $D$.
    The character of the standard permutation representation~$P$
    has the form  $(1_H)^{S_n} = 1_{S_n} + \chi$,
    where $\chi$ is an irreducible character of $S_n$.
    We are going to show that $\chi$ is the only 
    nontrivial irreducible constituent of~$\phi$.
    
    As we remarked in the first paragraph of the proof,
    we can assume that $\phi$ 
    does not contain the trivial character.
    The matrix $\sum_{g\in S_n} D(g)$ is fixed under
    multiplication with elements from $D(S_n)$,
    and since the trivial representation is not a constituent of~$D$,
    we have $\sum_{g\in S_n} D(g) = 0$.
    Geometrically, this means that
    the origin is the barycenter 
    of the representation polytope~$P(D)$.
    As $D(S_n\setminus H)$ is a facet of $P(D)$,
    we must have
      \[ \sum_{g\in S_n \setminus H} D(g) \neq 0
         ,\quad \text{and} \quad 
          \sum_{g\in  H} D(h) \neq 0.
      \]
    It follows that the restricted character
    $\phi_H$ contains the trivial character $1_H$
    as a constituent.
    Using Frobenius reciprocity and the fact that
    $(1_H)^{S_n} = 1_{S_n} + \chi$, we get 
    \[ 0 \neq [\phi_H, 1_H ] = [\phi, (1_H)^{S_n}]
       = [\phi,1_{S_n}] + [\phi,\chi] =  [\phi,\chi].  
    \]
    Thus $\chi$ is a constituent of $\phi$.

  Since dimension is a combinatorial invariant, we must have
  $\dim P(D) = \dim B_n = \chi(1)^2 $.
  On the other hand, we have
  $\dim P(D) = \sum_{\psi} \psi(1)^2$,
  where the sum runs over the nontrivial 
  irreducible constituents $\psi$ of $\phi$,
  not counting multiplicities~\cite[Theorem~3.2]{guralnickperkinson06}.
  It follows that $\chi$ is the only irreducible constituent of $\phi$,
  and thus $D$ and $P$ are stably equivalent.
\end{proof}

\begin{remark}
  In the preceding proof,
  we reduced to the case that the combinatorial isomorphism
  sends $P(g)$ to $D(g)$ (for any $g\in S_n$).
  If we could show that then $P(g)\mapsto D(g)$ 
  can be extended to an affine isomorphism, 
  Lemma~\ref{l:snbirk_unique} would follow
  from a characterization of effective equivalence
  by Baumeister and 
  Grüninger~\cite[Corollary~4.5]{BaumeisterGrueninger15}.
  But we do not know how to do this, or whether this is even true
  more generally 
  (for combinatorial isomorphisms of this form between
  representation polytopes of arbitrary groups).
\end{remark}

Finally, we prove our main result:

\begin{proof}[Proof of Theorem~\ref{main:birkhoff_unique}]
    Identify the vertices of $P(D)$ and $B_n$
    with $G$ and $S_n$, respectively.
    Let 
    $\gamma\colon G\to S_n$ be a combinatorial isomorphism.
    Then $\gamma$ induces an isomorphism $\kappa_{\gamma}$
    from the combinatorial symmetry group
    $A$ of $P(D)$ onto the combinatorial symmetry group
    $S_n\wr C_2$ of $B_n$ sending 
    $\alpha\in A$ to 
    $\kappa_{\gamma}(\alpha):=\gamma \circ \alpha \circ \gamma^{-1}$.
    Obviously,
    we have $\gamma(\alpha g) = \kappa_{\gamma}(\alpha)(\gamma g)$.
    Thus the pair
    $(\kappa_{\gamma}, \gamma)$ is an isomorphism
    from the $A$-set $G$ onto the 
    $(S_n\wr C_2)$-set $S_n$.    
    In particular, $\kappa_{\gamma}$ sends subgroups
    of $A$ which act regularly on $G$, onto subgroups
    of $S_n \wr C_2$ which act regularly on $S_n$.

    The left and right multiplications with elements of $G$
    induce regular subgroups of $A$.
    These are sent to regular subgroups
    $L$ and $R$ (say) of $S_n \wr C_2$.
    Since left and right multiplications centralize each other,
    the subgroups $L$ and $R$ centralize each other.
    If $n\geq 4$, then Lemma~\ref{l:snwrc2_reg} yields
    that $L= S_n \times 1$ or $L=1\times S_n$. 
    Since $L\iso G$, we have that $G\iso S_n$.
    In view of Lemma~\ref{l:snbirk_unique},
    this finishes the proof in case $n\geq 4$.
    
    In the case $n=3$, however, there is one additional
    possibility (up to conjugacy in $S_3 \wr C_2$),
    namely that 
    $L= R = C_2 \times C_3\iso C_6$.
    And indeed, the action of
    $C_2\times C_3$ on $\mat_3(\reals)$ yields the 
    Birkhoff polytope~$B_3$ as orbit polytope of $C_6$,    
    and this orbit polytope is affinely equivalent 
    to the representation polytope $P(D)$,
    where $D\colon C_6 \to \GL(4,\reals)$
    sends a generator of $C_6$ to
    \[ \begin{pmatrix*}[r]
         0 & 1 & & \\
        -1 & -1 & & \\
           &    & 0 & -1 \\
           &    & 1 & 1 
       \end{pmatrix*}.
    \]
\end{proof}  

\section*{Acknowledgments}
Part of the work was done while the second author visited 
Bielefeld University. 
We wish to thank the CRC 701
``\emph{Spectral Structures and Topological Methods in Mathematics}'' 
for its support.
The second author is also supported by the DFG 
through project SCHU 1503/6-1.

\printbibliography   


\end{document}